\documentclass[12pt]{amsart}

\usepackage{amsmath,amsfonts,amssymb,amscd,amsthm,amsbsy,graphics,epsf}
\usepackage{comment,graphicx}

\theoremstyle{plain}
\newtheorem{theorem}{Theorem}[section]

\theoremstyle{definition}

\numberwithin{equation}{section}

\makeatletter               
\let\c@equation\c@theorem       
\makeatother                

\begin{document}

\title[Pluripolarity of graphs of algebroid functions]{Pluripolarity of graphs of algebroid functions}

\author[Z. S.  Ibragimov]{Zafar S. Ibragimov}

\address{Department of Mathematics\\ Urgench State University\\ Urgench 220100, Uzbekistan}
\email{z.ibragim@gmail.com}
\urladdr{}

\date{May 3, 2010}

\thanks{}



\maketitle

Let $K\subset \mathbf C^n$ be an arbitrary compact set and let $f(z)$ be a continuous function on $K$. By $\rho_m(f,K)$ we denote the least
deviation of $f(z)$ on $K$ from the rational functions of degree less than or equal to $m$. That is,
$$
\rho_m(f,K)=\inf_{r_m}||f-r_m||_K,
$$
where $||-||_K$ is the uniform norm and the infimum is taken over all rational functions of the form
$$
r_m(z)=\frac {\underset{|\alpha|\leq m}\Sigma a_{\alpha} z^\alpha}{\underset{|\alpha|\leq m}\Sigma b_{\alpha} z^\alpha},\quad\text{where}\quad\alpha=(\alpha_1,\alpha_2,\dots,\alpha_n)\quad\text{is\ a\ multiindex}.
$$
As usual, we denote by $e_m(f,K)$ the least deviation of function $f(z)$ on $K$ from the polynomials of degree less than or equal to $m$. Obviously, $$
\rho_m(f,K)\leq e_m(f,K)\qquad\text{for\ each}\qquad m=1,2,\dots.
$$ 
In (\cite{G1}, \cite{G2}) Gonchar proved that if $K=[a,b]\subset\mathbf R\subset\mathbf C$, then the class of functions 
$$
R([a,b])=\{f\in C[a,b]\colon\, \underset{m\to\infty}{\underline{\lim}}\sqrt[m]{\rho_m(f,K)}<1\}
$$
possesses one of the important properties of the class of analytic functions. Namely, if 
$$
\underset{m\to\infty}{\underline{\lim}}\sqrt[m]{\rho_m(f,K)}<1\}
$$
and if $f(x)=0$ on a set $E\subset [a,b]$ of positive logarithmic capacity, then $f(x)\equiv 0$ on $[a,b]$ (see also \cite{P2}).

By analogy with the class  
$$
B(K)=\{f\in C(K)\colon\, \underset{m\to\infty}{\underline\lim}\sqrt[m]{e_m(f,K)}<1\},
$$
which is called the class of quasianalytic functions of Bernstein (\cite{T},\cite{P1},\cite{CLP}), we call 
$$
R(K)=\{f\in C(K)\colon\, \underset{m\to\infty}{\underline\lim}\sqrt[m]{\rho_m(f,K)}<1\}
$$
the class of quasianalytic functions of Gonchar. 

The classes $B(K)$ and $R(K)$ are not linear spaces; the sum of two quasianalytic functions are not, in general, quasianalytic (see [4]). We consider the following subclass $R_0(K)$ of the class $R(K)$:
$$
R_0(K)=\{f\in C(K)\colon\, \underset{m\to\infty}{\overline\lim}\sqrt[m]{\rho_m(f,K)}<1\}
$$ 
It is not hard to see that, if $f_1$ and $f_2$ belong to $R_0(K)$, then so are $c_1f_1+c_2f_2$ and $f_1\cdot f_2$, where $c_1$ and $c_2$ are arbitrary complex numbers.
	
In (\cite{DF}) K. Diederich and J.E. Fornass constructed an example of a smooth (infinitely differentiable) function, whose graph is not pluripolar in 
$\mathbf C^2$. Recently, D. Coman, N. Levenberg and E.A. Poletskiy have proved, that if $f\in B([a,b])$, then its graph $\Gamma_f$ is pluripolar in 
$\mathbf C^2$.
	
In (\cite{E}) A. Edigarian studied the following analogue of a theorem of N. Shcherbina [9]. Let $D$ be a domain in $\mathbf C^n$ and let $\Gamma\subset D\times\mathbf C$ be a graph of some algebroid function, i.e., 
$$
\Gamma=\Big\{(z,w)\in D\times\mathbf C\colon\, w^k+a_1(z)w^{k-1}+\dots +a_k(z)=0 \Big\},
$$
where $a_1(z), a_2(z),\dots, a_k(z)$ are continuous functions on $D$. Then $\Gamma$ is pluripolar in $\mathbf C^{n+1}$ if and only if the functions    $a_1(z), a_2(z),\dots, a_k(z)$ are holomorphic in $D$. In this paper we  prove a similar theorem on pluripolarity of graphs of algebroid functions in the class of quasianalytic functions.

\begin{theorem}
Let $[a,b]\subset\mathbf R\subset\mathbf C$ and let $\Gamma\subset\mathbf C^2$ be a graph of some algebroid function, i.e. 
$$
\Gamma=\Big\{(z,w)\in D\times\mathbf C\colon\, w^k+a_1(z)w^{k-1}+\dots +a_k(z)=0 \Big\},
$$
where $a_l\in R_0([a,b])$, $l=1,2,\dots,k$. Then $\Gamma$ is pluripolar in $\mathbf C^2$
\end{theorem}
\begin{proof}   We consider the following function on $[a,b]\times\mathbf C$
$$
f(z,w)=w^k+a_1(z)w^{k-1}+\dots +a_k(z).
$$
Since $a_l\in R_0([a,b])$ for $l=1,2,\dots,k$, there exist a sequence of rational functions $r^{j}_{1}(z), r^{j}_{2}(z),\dots,r^{j}_{m}(z),\dots$ such that  
$$
\sqrt[m]{\rho_m(a_j,[a,b])}=\sqrt[m]{||a_j(z)-r^{j}_{m}(z)||_{[a,b]}}\leq\delta_j<1.
$$ 
The function $f(z,w)$ is quasianalytic in the sense of Gonchar on the compact set $[a,b]\times\{|w|\leq h\}\subset\mathbf C^{k+1}$, where $h$ is an arbitrary positive number. Indeed, 
\begin{equation*}
\begin{split}
& \rho_{m+k}\big(f,[a,b]\times\{|w|\leq h\}\big)\leq ||f(z,w)-w^k-\sum_{j=1}^{k}r^{j}_{m}(z)w^{k-j}||_{[a,b]\times\{|w|\leq h\}}\\
& \leq\sum_{j=1}^{k}||a_j(z)-r^{j}_{m}(z)||_{[a,b]}\cdot h^j\leq\sum_{j=1}^{k}h^j\delta^{m+k}_{j}\leq k\max\{h, h^k\}\delta^{m+k},
\end{split}
\end{equation*}
where $\delta=\max\{\delta_j\colon\, j=1,2,\dots,k\}$. It follows that
$$
\underset{m\to\infty}{\overline\lim}\rho^{1/(m+k)}_{m+k}\big(f,[a,b]\times\{|w|\leq h\}\big)\leq\delta<1.
$$
Consequently, the graph 
$$
\Gamma_f=\big\{(z,w,f(z))\colon\, (z,w)\in [a,b]\times\mathbf C\big\}
$$
of the function $f(z,w)$ is pluripolar in $\mathbf C^3$. Now we consider sections
$$
\Gamma_f(\lambda)=\big\{(z,w,f(z,w))\colon\, f(z,w)=\lambda \big\}.
$$
For each $\lambda$, the section $\Gamma_f(\lambda)$  is pluripolar in $\mathbf C^2$. (Indeed, if the graph $\Gamma_f(\lambda)$ is nonpluripolar for some $\lambda\in\mathbf C$, then according to the uniqueness property of quasianalytic functions, the function $f(z,w)$ is identically equal to $\lambda$, which contradicts the definition of $f(z,w)$). In particular, we obtain the pluripolarity of the graph 
 $$
\Gamma=\Big\{(z,w)\in D\times\mathbf C\colon\, w^k+a_1(z)w^{k-1}+\dots +a_k(z)=0 \Big\}
$$
of algebroid functions. The proof of the theorem is complete. 
\end{proof}

\end{document}